 \documentclass[a4paper,reqno,12pt]{amsart} 
\usepackage{amsmath} 

\usepackage{amssymb}      

\usepackage{yhmath}
\usepackage{mathdots}
\usepackage{MnSymbol}

\usepackage{amsthm}      

\usepackage{tikz,amsmath}
\usetikzlibrary{arrows}

\usepackage{epsfig}      

\usepackage{graphicx}
\usepackage[normalem]{ulem}
\usepackage[all,line,
]{xy} 
\CompileMatrices 

\newcommand{\Ad}{\operatorname{Ad}}

\newcommand{\Span}{\operatorname{Span}}
\newcommand{\Tr}{\operatorname{Tr}}

   \theoremstyle{plain}
   \newtheorem{thm}{Theorem}[section]
   \newtheorem{prop}[thm]{Proposition}
   \newtheorem{lemma}[thm]{Lemma}  
   \newtheorem{cor}[thm]{Corollary}
   \theoremstyle{definition}

   \newtheorem{example}[thm]{Example}
   \theoremstyle{remark}

\newtheorem{ack}[thm]{Acknowledgement}

\usepackage{mdframed,xcolor}

\definecolor{mybgcolor}{gray}{0.8}
\definecolor{myframecolor}{rgb}{.647,.129,.149}

\mdfdefinestyle{mystyle}{
  usetwoside=false,
  skipabove=0.6em plus 0.8em minus 0.2em,
  skipbelow=0.6em plus 0.8em minus 0.2em,
  innerleftmargin=.25em,
  innerrightmargin=0.25em,
  innertopmargin=0.25em,
  innerbottommargin=0.25em,
  leftmargin=-.75em,
  rightmargin=-0em,
  topline=false,
  rightline=false,
  bottomline=false,
  leftline=false,
  backgroundcolor=mybgcolor,
  splittopskip=0.75em,
  splitbottomskip=0.25em,
  innerleftmargin=0.5em,
  leftline=true,
  linecolor=myframecolor,
  linewidth=0.25em,
}

\newmdenv[style=mystyle]{important}

\usepackage{lipsum}

   \numberwithin{equation}{section}

        \date{\today}

\title[Weights and K-theory]{On weights, traces and K-theory}
\author{Klaus Thomsen}

        \dedicatory{Til Palle E.T. J\o rgensen i taknemlighed}

\date{\today}

\bigskip

\bigskip

\address{Department of Mathematics, Aarhus University, Ny Munkegade, 8000 Aarhus C, Denmark}

\email{matkt@math.au.dk}

\begin{document}

\begin{abstract} It is shown that the pairing of the $K_{00}$ group of a $C^*$-algebra with the densely defined traces of the algebra can be extended to a pairing with the densely defined weights. For traces the pairing can be extended to the $K_0$ group without the semi-continuity assumption occurring in the work of Connes and Elliott. 

\end{abstract}

\maketitle
\section{Introduction} The pairing between the trace states and the $K_0$ group of a unital $C^*$-algebra is an important part of the invariant of $C^*$-algebras known as the Elliott invariant. I don't know who first observed the existence of this pairing, but it seems to be Connes who first observed that there is also a pairing with $K_0$ for non-unital $C^*$-algebras which uses traces that are only densely defined, but lower semi-continuous, \cite{Co1}. The latter pairing was revisited in a recent work by Elliott, \cite{E}. While Connes gives very few details of the construction, Elliott bases his approach on properties of the Pedersen ideal. The main purpose of the present text is to point out that the pairing exists also for traces that are not lower semi-continuous; it suffices that they are densely defined. This conclusion is obtained via a thorough examination of the constructions in \cite{Co1} and \cite{E}, and the original impetus for the present paper was the wish to point out that a condition, although natural from many points of view, is in fact not necessary. At first sight this may seem to have little or no significance because it can be difficult to come up with examples of densely defined traces that are not lower semi-continuous. Perhaps the only densely defined traces that come to mind and are not lower semi-continuous, are those coming from the Dixmier trace(s) on the $C^*$-algebra of compact operators. However, densely defined traces that are not lower semi-continuous exist in abundance, also on simple $C^*$-algebras very different from the $C^*$-algebra of compact operators. In an effort to convince the reader of this, the paper concludes by presenting constructions of such 'singular' traces on other $C^*$-algebras. It seems to me that it is worthwhile to pursue a more systematic investigation of such traces, and not only because they exist. But this is not the time or place.

The definition of the $K_0$-group $K_0(A)$ of a $C^*$-algebra $A$ starts with the formation of the group $K_{00}(A)$, cf. e.g. \cite{B}, and similarly the first step in the construction of the pairing between $K_0$ and traces is a pairing with $K_{00}(A)$. As a second purpose with this text, it is pointed out in the following first sections that there is a natural way to pair positive linear functionals, and more generally densely defined weights on $A$ with $K_{00}(A)$ in a way which generalizes the pairing with traces.

\begin{ack} I am grateful to the referees for their thoughtful and appropriate remarks and corrections. In particular, one of their remarks resulted in the deletion of an unnecessary assumption in Proposition \ref{02-01-23}.
\end{ack}
\section{Weights}

 Let $A$ be a $C^*$-algebra and let $A^+$ denote the cone of positive elements in $A$. A map $\psi : A^+ \to [0,\infty]$ is a \emph{weight} on $A$ when
\begin{itemize}
\item[(a)] $\psi(a+b) = \psi(a) + \psi(b), \ \ \forall a,b \in A^+$, and
\item[(b)] $\psi( t a) = t \psi(a), \ \ \forall a \in A^+, \ \forall t \in \mathbb R^+$, with the convention that $0 \cdot \infty = 0$.
\end{itemize}
A weight $\psi$ is \emph{densely defined} when 
$\left\{ a \in A^+ :   \psi(a) < \infty\right\}$ is dense in $A^+$ and a \emph{trace} when $\psi(aa^*) = \psi(a^*a)$ for all $a \in A$. A weight $\psi$ on $A$ is \emph{lower semi-continuous} when $\left\{a \in A^+ : \ \psi(a) > t\right\}$ is open in $A^+$ for all $t \in \mathbb R$. By a fundamental result of Combes, \cite{C}, this happens (if and) only if there is a set $\mathcal F$ of positive linear functionals on $A$ such that
$$
\psi(a) = \sup_{\omega \in \mathcal F} \omega(a), \ \ \forall a \in A^+.
$$
Given a weight $\psi$ on $A$ we set
$$
\mathcal M^+_\psi := \left\{ a \in A^+: \ \psi(a) < \infty\right\},
$$ 
$$
\mathcal N_\psi := \left\{ a \in A: \ \psi(a^*a) < \infty\right\},
$$
and
$$
\mathcal M_\psi := \Span \mathcal M^+_\psi.
$$
Based on work by Pedersen in \cite{Pe}, the following was observed by Combes in Lemme 1.1 and Lemme 1.3 of \cite{C}.

\begin{lemma}\label{04-11-21nx} Let $\psi$ be weight on $A$. The sets $\mathcal M_\psi$, $\mathcal M_\psi^+$ and $\mathcal N_\psi$ have the following properties.
\begin{itemize}
\item[(a)] $\mathcal N_\psi$ is a left-ideal in $A$; in particular, a subspace of $A$.
\item[(b)] $\mathcal M_\psi = \Span \mathcal N_\psi^* \mathcal N_\psi$ is a $*$-subalgebra of $A$.
\item[(c)]  $\mathcal M_\psi \cap A^+ = \mathcal M_\psi^+$.
\item[(d)] $\psi : \mathcal M^+_\psi \to [0,\infty)$ extends uniquely to a linear map $\psi|_{\mathcal M_\psi} : \mathcal M_\psi \to \mathbb C$.
\end{itemize}
\end{lemma}

\section{Pairing weights with $K_{00}$} 

The following brief description of the $K_0$ group of a $C^*$-algebra $A$ serves primarily to introduce the notation we shall use. For more details, see Blackadars book \cite{B} or R\o rdam et al. \cite{RLL}, for example.
Let $M_k(A)$ denote the $C^*$-algebra of $k \times k$ matrices over $A$ and let $P_k(A)$ denote the set of projections in $M_k(A)$. We embed $M_k(A)$ into $M_{k+1}(A)$ by adding a zero row and a zero column, and we can therefore consider the unions
$$
M_\infty(A) := \bigcup_{k=1}^\infty M_k(A) ,
$$
and
$$
P_\infty(A) := \bigcup_{k=1}^\infty P_k(A) .
$$
The Murray-von Neumann equivalence relation $\sim$ is defined in $P_\infty(A)$ such that $p \sim q$ when there is an element $v \in M_\infty(A)$ with $vv^* = p$ and $v^*v = q$. The equivalence class of an element $e \in P_\infty(A)$ is written $[e]$. Given a pair of projections $p,q \in P_\infty(A)$ there is a projection $p' \in P_{\infty}(A)$ such that $p' \perp q$, i.e. such that $p'q = 0$, and such that $p \sim p'$. The sum $p'+q$ is then a projection, and the sum $[p]+[q]$ is defined by
$$
[p]+[q] := [p'+q] .
$$
This turns $V(A) := P_\infty(A)/\sim$ into an abelian semi-group. The group $K_{00}(A)$ is the Grothendieck group of $V(A)$. To obtain the $K_0$ group, add a unit to $A$ to get the unital $C^*$-algebra $A^\dagger$ and define $K_0(A)$ as the kernel of the map
$$
K_{00}(A^\dagger) \to K_{00}(\mathbb C)
$$
induced by the quotient map $A^\dagger \to \mathbb C$. Since $[e]$ is an element of this kernel when $e \in P_\infty(A) \subseteq P_\infty(A^\dagger)$, there is a canonical map
\begin{equation}\label{21-10-22}
K_{00}(A) \to K_0(A)
\end{equation}
which is an isomorphism when $A$ is unital or at least has an approximate unit consisting of projections, cf. \cite{B}.


Let $\psi$ be a weight on the $C^*$-algebra $A$. For each $k \in \mathbb N$ we define a weight $\psi_k$ on $M_k(A)$ such that
$$
\psi_k(x) := \sum_{i=1}^k \psi(x_{ii}) = \psi(\sum_{i=1}^k x_{ii})
$$
when $x =(x_{ij}) \in M_k(A)^+$. This definition is compatible with the embedding of $M_k(A)$ into $M_{k+1}(A)$ and we get a map
$$
\psi_\infty : P_\infty(A) \to  [0,\infty]
$$
whose restriction to $P_k(A)$ agrees with $\psi_k$ for all $k$.

\begin{lemma}\label{26-09-22x}  For each $k \in \mathbb N$,
\begin{itemize}
\item[(e)] $\mathcal N_{\psi_k} = M_k(\mathcal N_\psi)$, and
\item[(f)]  $\mathcal M_{\psi_k} = M_k(\mathcal M_\psi)$.
\end{itemize}
\end{lemma}
\begin{proof} (e) follows from the observation that
$$
 \psi_k(x^*x) = \sum_{i,j} \psi(x_{ji}^*x_{ji})
$$
when $ x = (x_{ij}) \in M_k(A)$. To prove (f) note that it follows from (e) and (b) that 
$$
\mathcal M_{\psi_k} =\Span \mathcal N_{\psi_k}^*\mathcal N_{\psi_k} \subseteq \Span M_k(\mathcal N_\psi^* \mathcal N_\psi) \subseteq M_k(\mathcal M_\psi). 
$$
For the converse inclusion, consider $a,b \in \mathcal N_\psi$ and the standard matrix units $\{e_{ij}\}$ in $M_k(\mathbb C)$. Then $a^*b \otimes e_{ij} = (a \otimes e_{1i})^*(b \otimes e_{1j})$. Note that $(a \otimes e_{1i})^*(a \otimes e_{1i}) = a^*a\otimes e_{ii} \in \mathcal M_{\psi_k}^+$ and similarly $(b \otimes e_{1j})^*(b \otimes e_{1j}) = b^*b\otimes e_{jj} \in \mathcal M_{\psi_k}^+$. Hence $a \otimes e_{1i} \in \mathcal N_{\psi_k}$ and $b \otimes e_{1j} \in \mathcal N_{\psi_k}$, and thus $a^*b \otimes e_{ij} \in \mathcal N_{\psi_k}^*\mathcal N_{\psi_k} \subseteq \mathcal M_{\psi_k}$ by (b). Since $i,j,a,b$ were arbitrary it follows now also from (b) that $M_k(\mathcal M_\psi) = M_k\left(\Span \mathcal N_\psi^*\mathcal N_\psi\right) \subseteq \mathcal M_{\psi_k}$.
\end{proof}

\begin{lemma}\label{29-09-22x} Let $\psi$ be a densely defined weight on $A$. Then $\psi_k$ is densely defined for all $k \in \mathbb N$.
\end{lemma}
\begin{proof} Let $x \in M_k(A)^+$. Write $x = a^2$ where $a \in  M_k(A)^+$. Since every element of $A$ is a linear combination of four elements of $A^+$ it follows that $\mathcal M_\psi$ is dense in $A$ since $\psi$ is densely defined. It follows from (a) and (b) that $\mathcal M_\psi \subseteq \mathcal N_\psi$ and hence $\mathcal N_\psi$ is also dense in $A$. It follows therefore that $M_k(\mathcal N_\psi)$ is dense in $M_n(A)$. In particular, there is a sequence $\{a_n\}$ in $M_k(\mathcal N_\psi)$ such that $\lim_{n \to \infty} a_n = a$. Note that $a_n^*a_n \in \mathcal N_{\psi_k}^*\mathcal N_{\psi_k}$ by (e) and hence $a_n^*a_n \in \mathcal M_{\psi_k}^+$ by (b) and (c). Since $\lim_{n \to \infty} a_n^*a_n =x$, we have shown that $\psi_k$ is densely defined. 
\end{proof}

\begin{lemma}\label{24-09-22x} Let $\psi$ be a densely defined weight on $A$ and let $e \in P_k(A)$. There is a projection $f \in P_k(A)$ such that $e \sim f$ and $\psi_k(f) < \infty$.
\end{lemma}
\begin{proof} Let $\delta \in ]0,\frac{1}{2}[$ and define $f: [0,1] \to [0,1]$ such that $f$ is continuous and
$$
f(t) = \begin{cases} 0, & \ t \in [0,\delta] \\ \text{linear}, & \ t \in [\delta,1-\delta] \\ 1, & \ t \in [1-\delta,1]. \end{cases}
$$
Then 
\begin{equation}\label{28-09-22ax}
\left|f(t) -t\right| \leq \delta \ \ \forall t \in [0,1] .
\end{equation} 
$\psi_k$ is densely defined by Lemma \ref{29-09-22x} and we can therefore find $a \in \mathcal M_{\psi_k}^+$ such that $0 \leq a \leq 1$ and $\|a - e\|$ is as small as we want. In particular, we can arrange that the spectrum $\sigma(a)$ of $a$ is contained in $[0,\delta] \cup [1-\delta,1]$ and $\|a-e\| \leq \delta$. It follows from \eqref{28-09-22ax} that $\left\|f(a) - a \right\| \leq \delta$. Note that $f(a)$ is a projection in $M_k(A)$ and that $\|f(a) - e \| \leq 2 \delta$. Since $2 \delta < 1$ it follows that $f(a) \sim e$, cf. e.g. Proposition 4.6.6 in \cite{B}. Since $f(t) \leq (1-\delta)^{-1} t$ for all $t \in \sigma(a)$, it follows that $f(a) \leq (1-\delta)^{-1} a$ and hence $\psi_k(f(a)) \leq (1-\delta)^{-1}\psi_k(a) < \infty$. 

\end{proof}

 For $e \in P_\infty(A)$, set
$$
\underline{\psi}(e) := \inf \left\{\psi_\infty(f): \ f \in P_\infty(A), \ f \sim e \right\} .
$$
Then Lemma \ref{24-09-22x} has the following corollary.

\begin{cor}\label{24-09-22a}  Let $\psi$ be a densely defined weight on $A$. Then $\underline{\psi}(e) < \infty$ for all $e \in P_\infty(A)$.
\end{cor}

\begin{lemma}\label{24-09-22bx}  Let $\psi$ be a densely defined weight on $A$. Let $e,f \in P_\infty(A)$ such that $ef = 0$. Then $\underline{\psi}(e+f) = \underline{\psi}(e) + \underline{\psi}(f)$.
\end{lemma}
\begin{proof} When $p \in P_\infty(A)$ and $p \sim e + f$ we can write $p = p_1+p_2$ where $p_1 \sim e$ and $p_2 \sim f$. Then $\psi_\infty(p) = \psi_\infty(p_1) + \psi_\infty (p_2) \geq \underline{\psi}(e)+ \underline{\psi}(f)$. It follows that $\underline{\psi}(e+f) \geq \underline{\psi}(e) + \underline{\psi}(f)$. To obtain the reverse inequality, let $\epsilon > 0$. There is a $k \in \mathbb N$ such that $e,f \in P_k(A)$, and by increasing $k$ if necessary we can also arrange that there are projections $p,q \in P_k(A)$ such that $e \sim p, \ f \sim q$, $\psi_k(p) \leq \underline{\psi}(e)+\epsilon$ and $  \psi_k(q) \leq \underline{\psi}(f)+ \epsilon$. There is then a unitary $S \in M_{2k}(\mathbb C)$ such that $(1_A \otimes S)p(1_A \otimes S)^* q = 0$. Then 
\begin{equation}\label{26-09-22ax}
e + f \sim (1_A \otimes S)p(1_A \otimes S)^* + q. 
\end{equation}
We claim that
\begin{equation}\label{24-09-22cx}
\psi_{2k}\left(  (1_A \otimes S)p(1_A \otimes S)^*\right) = \psi_k(p).
\end{equation}
To see this, note that since $p \in \mathcal M_{\psi_{2k}}^+$ it follows from (f) of Lemma \ref{26-09-22x} that $p \in  M_{2k}(\mathcal M_\psi) = \mathcal M_\psi \otimes M_{2k}(\mathbb C)$. Thanks to (d) and (f) it makes sense to assert that
\begin{equation}\label{28-09-22x}
\psi_{2k}|_{\mathcal M_{\psi_{2k}}}= \psi|_{\mathcal M_\psi} \otimes \Tr_{2k} ,
\end{equation}
where $\Tr_{2k}$ is the standard trace on $M_{2k}(\mathbb C)$; the sum of the diagonal entries. Furthermore, the identity \eqref{28-09-22x} can be easily verified by checking on simple tensors. Since $\Ad (1_A \otimes S)(M_{2k}(\mathcal M_\psi)) = M_{2k}(\mathcal M_\psi)$, it follows from this that
\begin{align*}
& \psi_{2k}\left(  (1_A \otimes S)p(1_A \otimes S)^*\right) = \psi_{2k} \circ \Ad (1_A \otimes S)(p)\\
& = (\psi|_{\mathcal M_\psi} \otimes \Tr_{2k} )\circ \Ad (1_A \otimes S)(p) \\
& = \left(\psi|_{\mathcal M_\psi} \otimes (\Tr_{2k} \circ \Ad S)\right)(p)\\
& = \left(\psi|_{\mathcal M_\psi} \otimes \Tr_{2k}\right)(p) = \psi_{2k}(p) = \psi_k(p),
\end{align*}
proving the claim \eqref{24-09-22cx}. It follows that
$$
\psi_\infty( (1_A \otimes S)p(1_A \otimes S)^* + q) = \psi_\infty(p) + \psi_\infty(q) \leq \underline{\psi}(e) + \underline{\psi}(f) + 2\epsilon,
$$
and then \eqref{26-09-22ax} implies that $\underline{\psi}(e+f) \leq \underline{\psi}(e) + \underline{\psi}(f) + 2\epsilon$. Since $\epsilon > 0$ was arbitrary it follows that $\underline{\psi}(e+f) \leq \underline{\psi}(e) + \underline{\psi}(f)$.
\end{proof}

The following is now an immediate consequence of the definitions and Corollary \ref{24-09-22x} and Lemma \ref{24-09-22bx}.

\begin{prop}\label{24-09-22dx} Given a densely defined weight $\psi$ on $A$, there is a homomorphism $\psi_* : K_{00}(A) \to \mathbb R$ such that
$$
\psi_*([e] -[f]) = \underline{\psi}(e) - \underline{\psi}(f)
$$
when $e,f \in P_\infty(A)$.
\end{prop}

When the weight $\psi$ is a trace we have that $\underline{\psi}(e) = \psi_\infty(e)$ for all $e \in P_\infty(A)$ and the homomorphism $\psi_*$ of Proposition \ref{24-09-22dx} is therefore the usual one, but note that we do not need to assume the trace is lower semi-continuous. When $A$ has an approximate unit consisting of projections there is an identification $K_{00}(A) = K_0(A)$, \cite{B}, and Proposition \ref{24-09-22dx} gives then a pairing of densely defined weights on $A$ with $K_0(A)$. In particular, when $A$ is unital we get a map from the states of $A$ to the states on $K_0(A)$ extending the map from trace states on $A$ which occurs in the Elliott invariant.

\begin{example} Let $H$ be a Hilbert space and $h $ a bounded positive operator on $H$. We can then consider the densely defined weight $\psi$ on the $C^*$-algebra $\mathbb K$ of compact operators on $H$ given by
$$
\psi(a) = \Tr (ha), \ \ \forall a \in \mathbb K^+ .
$$
This is an unbounded weight unless $h$ is of trace class. Using the identification $K_{00}(\mathbb K) = \mathbb Z$ the homomorphism $\psi_*$ of Proposition \ref{24-09-22dx} is given by
$$
\psi_*(z) = \lambda z, \ \ \forall z \in \mathbb Z,
$$
where 
$$
\lambda = \inf \left\{\left< h\phi,\phi\right>: \ \phi \in H, \|\phi\| = 1 \right\} 
$$
is the minimum of the spectrum of $h$.  
\end{example}

\section{Pairing traces with $K_0$}

In the following we fix a densely defined trace $\tau: A^+ \to [0,\infty]$ on $A$.

\begin{lemma}\label{16-11-22} $\mathcal M_\tau$ is a dense two-sided $*$-invariant ideal in $A$.
\end{lemma}
\begin{proof} $\mathcal N_\tau$ is a left ideal by (a) of Lemma \ref{04-11-21nx}. The trace property of $\tau$ implies that $\mathcal N_\tau^* = \mathcal N_\tau$ and hence $\mathcal N_\tau$ is also a right ideal. It follows therefore from (b) of Lemma \ref{04-11-21nx} that $\mathcal M_\tau$ is a $*$-invariant two-sided ideal in $A$. It is dense in $A$ because $\mathcal M_\tau^+$ is dense in $A^+$ by assumption. 
\end{proof}

Let $A^\dagger$ be the $C^*$-algebra obtained from $A$ be adjoining a unit to $A$. Thus as a vector space $A^\dagger$ is just $A \oplus \mathbb C$, and the product and involution are given by
$$
(a,\lambda)(b,\mu) = (ab + \lambda b + \mu a, \lambda \mu) 
$$
and
$$
(a,\lambda)^* = (a^*, \overline{\lambda}) .
$$
For every subset $X \subseteq A$, set
$$
X^\dagger := \left\{(a,\lambda) \in A^\dagger: \ a \in X, \ \lambda \in \mathbb C \right\}.
$$
In particular,
$$
\mathcal M_\tau^\dagger := \left\{(a,\lambda) \in A^\dagger: \ a \in \mathcal M_\tau, \ \lambda \in \mathbb C \right\} .
$$
Define $\tau^\dagger : \mathcal M_\tau^\dagger \to \mathbb C$ such that
$$
\tau^\dagger(a,\lambda) := \tau|_{\mathcal M_\tau}(a) .
$$
Since $\tau^\dagger$ is linear on $\mathcal M_\tau^\dagger$ we can consider the tensor product map
$$
\tau^\dagger \otimes \Tr_n : \  M_n(\mathcal M_\tau^\dagger)\to \mathbb C,
$$
where $\Tr_n$ denotes the standard trace on $M_n(\mathbb C)$. Let $P_n(\mathcal M_\tau^\dagger)$ be the set of projections in the $*$-algebra $ M_n(\mathcal M_\tau^\dagger)$ and set
$$
 P_\infty(\mathcal M_\tau^\dagger) :=  \bigcup_n P_n(\mathcal M_\tau^\dagger) .
 $$
We aim to establish the following

\begin{thm}\label{19-11-22d} 
$$K_0(A^\dagger) = \left\{ [e] - [f] : \ e,f \in P_\infty(\mathcal M_\tau^\dagger) \right\}
$$
and there is a homomorphism $\tau^\dagger_* : \ K_0(A^\dagger) \to \mathbb R$ such that
$$
\tau^\dagger_*([e]-[f]) = \tau^\dagger \otimes \Tr_n(e)  - \tau^\dagger \otimes \Tr_n(f)
$$
when $e,f \in M_n(\mathcal M_\tau^\dagger)$.
\end{thm}

\subsection{Proof of Theorem \ref{19-11-22d}} The proof uses the following series of lemmas.

Recall that a \emph{strictly positive} element of a $C^*$-algebra $A$ is a positive element $a \in A^+$ such $\omega(a) > 0$ for all non-zero positive functionals $\omega$ on $A$. A $C^*$-algebra is \emph{$\sigma$-unital} when it contains a strictly positive element. The first two lemmas are well-known. Their proofs are included for completeness.

\begin{lemma}\label{02-01-23x} A separable $C^*$-algebra is $\sigma$-unital.
\end{lemma}
\begin{proof} When $A$ is a separable $C^*$-algebra there is a dense sequence $\{a_n\}_{n=1}^\infty$ in $\{a \in D^+: \ 0 \leq a \leq 1 \}$. Then
$$
a := \sum_{n=1}^\infty 2^{-n}a_n
$$
is strictly positive.
\end{proof}

\begin{lemma}\label{19-11-22fx} Let $D$ be a $\sigma$-unital $C^*$-algebra. There is a sequence $\{d_n\}_{n=1}^\infty$ in $D$ such that
\begin{itemize}
\item $0 \leq d_n \leq 1, \ \ \forall n$,
\item $d_nd_{n+1} = d_n, \ \ \forall n$, and
\item $\lim_{n \to \infty} d_na =a, \ \ \forall a \in D$.
\end{itemize}
\end{lemma}

\begin{proof} Let $a_0$ be a strictly positive element of $A$. Let $f_n$ be the continuous function $f_n : [0,\infty) \to [0,1]$ such that
$$
f_n(t) = \begin{cases} 0, & \ t \in [0,\frac{1}{n+1}]\\ \text{linear}, & \ t \in \left[\frac{1}{n+1}, \frac{1}{n}\right], \\ 1, & \ t \geq \frac{1}{n} . \end{cases}
$$
Set $d_n := f_n(a_0)$. The first item holds since $0 \leq f_n \leq 1$ and the second because $f_nf_{n+1} = f_n$. To establish the third, assume for a contradiction that there is an element $x \in D$ for which $(1-f_n(a_0))x$ does not converge to $0$. Then 
$$
\sup_{\omega } \omega((1-f_n(a_0))xx^*(1-f_n(a_0))
$$
does not converge to $0$ when we take the supremum over all states $\omega$ of $A$. It follows that there is an $\epsilon > 0$ and a sequence $n_1 < n_2 < n_3 < \cdots$ in $\mathbb N$ such that for each $k$ there is a state $\omega_k$ with
$$
\omega_k((1-f_{n_k}(a_0))xx^*(1-f_{n_k}(a_0))) \geq \epsilon .
$$
Since the unit ball of $A^*$ is weak* compact there is a weak* condensation point $\mu$ of the set of functionals defined by
$$
A\ni a \mapsto  \omega_k((1-f_{n_k}(a_0))a (1-f_{n_k}(a_0))), \  \  k \in \mathbb N.
$$
Then $\mu(xx^*) \geq \epsilon$, and hence $\mu$ is a non-zero positive functional on $A$. However, $\lim_{n \to \infty} (1-f_n(t))t =0$ uniformly for $t$ in the spectrum of $a_0$, and hence $\lim_{n \to \infty} (1-f_n(a_0))a_0 = 0$. It follows that $\mu(a_0) = 0$, contradicting the strict positivity of $a_0$.  
\end{proof}

Consider now a separable $C^*$-subalgebra $D$ of $A$, and let
$$
{\bf d} := \{d_n\}_{n=1}^\infty
$$
be a sequence in $D$ with the properties specified in Lemma \ref{19-11-22fx}. Set
$$
D_{d_n} := \left\{ a \in D : \ ad_n = d_na = a \right\} .
$$
Then $D_{d_n}$ is a $C^*$-subalgebra of $D$ and 
$$
D_{d_n} \subseteq d_nDd_n \subseteq D_{d_{n+1}} \subseteq d_{n+1}Dd_{n+1}.
$$ 
In particular,
$$
\bigcup_{n=1}^\infty d_nDd_n = \bigcup_{n=1}^\infty D_{d_n} .
$$
Set
$$
D({\bf d}) := \bigcup_{n=1}^\infty D_{d_n} .
$$
Then $D({\bf d})$ is a $*$-subalgebra of $D$ and it is dense in $D$ since 
$$
\lim_{n \to \infty} d_nad_n =a
$$ 
for all $a\in D$.

\begin{lemma}\label{18-11-22gx} $D({\bf d})\subseteq \mathcal M_\tau$.
\end{lemma}
\begin{proof} The key step in the proof comes from the proof of Theorem 1.3 in \cite{Pe} where Pedersen introduces his famous minimal dense ideal. Let $n \in \mathbb N$ and consider an element $a \in D^+$. Set $d := d_nad_n$. It suffices to show that $d \in \mathcal M_\tau$. Since $\tau$ is densely defined there is an $x \in A^+ \cap \mathcal M_\tau$ such that $\left\|x -d_{n+1}\right\| \leq \frac{1}{2}$.  
Using $d_{n+1}\sqrt{d} =\sqrt{d}$ we find that
$$
\frac{1}{2} d \leq \sqrt{d}(1+(x-d_{n+1}))\sqrt{d} = \sqrt{d}(d_{n+1}+(x-d_{n+1}))\sqrt{d} = \sqrt{d}x\sqrt{d}.
$$
Note that $\sqrt{d}x\sqrt{d} \in \mathcal M_\tau$ since $\mathcal M_\tau$ is a two-sided ideal by 
Lemma \ref{16-11-22}. It follows therefore from the estimate above that $\tau(d) \leq 2 \tau(\sqrt{d}x\sqrt{d}) < \infty$. Hence $d \in \mathcal M_\tau$.
\end{proof}

\begin{lemma}\label{17-11-22bx} $\tau^\dagger(xy) = \tau^\dagger(yx)$ for all $x,y \in \mathcal M_\tau^\dagger$.
\end{lemma}
\begin{proof} Write $x = (a,\lambda), \ y = (b,\mu)$ where $a,b \in \mathcal M_\tau, \ \lambda,\mu \in \mathbb C$. Then $\tau^\dagger(xy) = \tau|_{\mathcal M_\tau}(ab + \lambda b + \mu a)$ while $\tau^\dagger(yx) = \tau|_{\mathcal M_\tau}(ba + \lambda b + \mu a)$. It suffices therefore to show that $\tau|_{\mathcal M_\tau}(ab) = \tau|_{\mathcal M_\tau}(ba)$, which follows from the polarization identities
$$
ab = \frac{1}{4}\sum_{k=1}^4 i^k(b + i^ka^*)^*(b + i^ka^*)
$$
and
$$
ba = \frac{1}{4}\sum_{k=1}^4 i^k(b + i^ka^*)(b + i^ka^*)^* .
$$
\end{proof}

\begin{cor}\label{20-11-22x}$\tau^\dagger \otimes \Tr_n(xy) = \tau^\dagger \otimes \Tr_n(yx)$ for all $x,y \in M_n(\mathcal M_\tau^\dagger)$.
\end{cor}

\begin{lemma}\label{04-11-22ax} Let $I$ be a right ideal (not necessarily closed) in the $C^*$-algebra $A$. If $x \in M_n(I^\dagger)$ is invertible in $M_n(A^\dagger)$ then $x^{-1} \in M_n(I^\dagger)$.
\end{lemma}
\begin{proof} Write $x = (a,k)$ and $x^{-1} = (b,m)$ where $a \in M_n(I)$, $b \in M_n(A)$ and $k,m \in M_n(\mathbb C)$. Since $xx^{-1} = 1$ we get the equations $ab + kb+am =0 $ and $km =1$, implying that $b = k^{-1}(-am-ab) \in M_n(I)$.
\end{proof}

\begin{lemma}\label{19-11-22c} Let $e,f$ be projections in $M_n(\mathcal M_\tau^\dagger)$, and assume that there is a partial isometry $v \in M_n(A^\dagger)$ such that $e = vv^*$ and $f= v^*v$. Then $\tau^\dagger \otimes \Tr_n(e) = \tau^\dagger \otimes \Tr_n(f)$.
\end{lemma}
\begin{proof} The proof is based on material from Section II. 4 of \cite{B}. Choose a separable $C^*$-algebra $D \subseteq A$ such that $e,f,v \in M_n(D^\dagger)$ and choose in $D$ a sequence $\{d_n\}_{n=1}^\infty$ with the properties specified in Lemma \ref{19-11-22fx}. For each $\epsilon >0$ we can then find $N \in \mathbb N$ and projections $e',f' \in M_n(D_{d_N}^\dagger)$ such that $\left\| e-e'\right\| \leq \epsilon$ and $\|f-f'\| \leq \epsilon$, cf. Lemma  6.3.1 of \cite{RLL}. Set $u := (2e'-1)(2e-1) + 1$. Note that
\begin{align*}
&\left\|1- \frac{1}{2}u\right\| = \left\|(2e'-1)(e'-e)\right\| \leq \|e'-e\|,  
\end{align*} 
implying that $u$ is invertible if $\epsilon < 1$. Since $ue = 2e'e = e'u$, we see that $ueu^{-1} = e'$. Note that $u \in M_n(\mathcal M_\tau^\dagger)$. It follows therefore from Lemma \ref{04-11-22ax} that $u^{-1} \in M_n(\mathcal M_\tau^\dagger)$. Since $e' \in M_n(\mathcal M_\tau^\dagger)$ by Lemma \ref{18-11-22gx} it follows now from Corollary \ref{20-11-22x} that $\tau^\dagger \otimes \Tr_n(e') = \tau^\dagger \otimes \Tr_n(e)$. We observe that $u\in M_n(D^\dagger)$, which implies that $u \in M_n((\mathcal M_\tau \cap D)^\dagger)$ and hence that $u^{-1} \in M_n((\mathcal M_\tau \cap D)^\dagger)$ by Lemma \ref{04-11-22ax}. Similarly, $\tau^\dagger \otimes \Tr_n(f') = \tau^\dagger \otimes \Tr_n(f)$ and $sfs^{-1} = f'$ for some invertible element $s \in M_n((\mathcal M_\tau \cap D)^\dagger)$ with $s^{-1} \in M_n((\mathcal M_\tau \cap D)^\dagger)$. We aim to show that $\tau^\dagger \otimes \Tr_n(e') = \tau^\dagger \otimes \Tr_n(f')$. Set $x:= f'sv^*u^{-1}e', \ y := e'uvs^{-1}f'$. Then $x,y \in M_n(D^\dagger)$ and
\begin{align*}
&xy = f'sv^*u^{-1}e'uvs^{-1}f' = f'sv^*evs^{-1}f' \\
&=  f'sv^*vv^*vs^{-1}f'=f'sfs^{-1}f' =f'.
\end{align*}
Similarly, $yx =e'$. Since $f' = {f'}^*f' = y^*x^*xy \leq \|x\|^2 y^*y$ it follows that $y^*y$ is positive and invertible in $f'M_n(D^\dagger)f'$. Taking the inverse in that algebra, set 
$$
w:= y(y^*y)^{-1/2} .
$$ 
Then $w^*w = (y^*y)^{-1/2}y^*y(y^*y)^{-1/2} = f'$ and hence, in particular, $ww^*$ is a projection. Note that
\begin{align*}
&e' = yxx^*y^* \leq \|x\|^2 yy^* = \|x\|
^2 y(y^*y)^{-1/2}y^*y(y^*y)^{-1/2}y^* \\
&\leq \|x\|^2\|y\|^2  y(y^*y)^{-1/2}(y^*y)^{-1/2}y^* = \|x\|^2\|y\|^2 ww^* .
\end{align*}
Then 
$$
0 \leq (1-ww^*)e'(1-ww^*) \leq \|x\|^2\|y\|^2(1-ww^*)ww^*(1-ww^*) = 0,
$$ 
implying that $e' = ww^*e'$. On the other hand, since $e'y = y$, it follows that $e'ww^* = ww^*$ and hence $e' = ww^*$. Since $w \in M_n(D^\dagger)$ and $D({\bf d})$ is dense in $D$ there is a $k > N$ and an element $z \in M_n(D_{d_k}^\dagger)$ such that $\left\|z{z}^* - e'\right\| < \frac{1}{2}$ and $\left\|{z}^*z - f'\right\| < \frac{1}{2}$. Since $e',f',z \in M_n(D_{d_k}^\dagger)$ it follows then from Lemma 6.3.1 in \cite{RLL} that there is a partial isometry $w' \in M_n(D_{d_k}^\dagger)$ such that $w'{w'}^* = e'$ and ${w'}^*w' = f'$. Since $w' \in M_n(\mathcal M_\tau^\dagger)$ by Lemma \ref{18-11-22gx} it follows from Corollary \ref{20-11-22x} that $ \tau^\dagger \otimes \Tr_n(e') = \tau^\dagger \otimes \Tr_n(f')$. 
\end{proof}

\begin{lemma}\label{20-11-22a} Let $e \in M_n(A^\dagger)$ be a projection. There is a projection $e' \in M_n(\mathcal M_\tau^\dagger)$ and a partial isometry $v \in M_n(A^\dagger)$ such that $vv^*= e$ and $v^*v= e'$.
\end{lemma}
\begin{proof} Let $D \subseteq A$ be a separable $C^*$-subalgebra such that $e \in M_n(D^\dagger)$ and let $\{d_n\}_{n=1}^\infty$ be a sequence in $D$ with the properties specified in Lemma \ref{19-11-22fx}. Since $D({\bf d})$ is a $*$-algebra which is dense in $D$ and since $M_n(D_{d_k}^\dagger)$ is a $C^*$-algebra it follows from Lemma 6.3.1 in \cite{RLL} that there is a $k$, a projection $e' \in M_n(D_{d_k}^\dagger)$ and a partial isometry $v \in M_n(A^\dagger)$ such that $vv^*= e$ and $v^*v= e'$. This completes the proof because $M_n(D_{d_k}^\dagger) \subseteq M_n(\mathcal M_\tau^\dagger)$ by Lemma \ref{18-11-22gx}. 
\end{proof}

\emph{Proof of Theorem \ref{19-11-22d}:} The identity 
$$
K_0(A^\dagger) = \left\{ [e] - [f] : \ e,f \in P_\infty(\mathcal M_\tau^\dagger) \right\}
$$ follows from the definition of $K_0(A^\dagger)$ and Lemma \ref{20-11-22a}. To prove that $\tau^\dagger_*$ is well-defined assume that $e,f,e',f'$ are projections in $M_n(\mathcal M_\tau^\dagger)$ such that $[e] - [f] = [e']-[f']$ in $K_0(A^\dagger)$. There is then a projection $r \in P_\infty(A^\dagger)$ such that $e\oplus f' \oplus r$ and $e'\oplus f \oplus r$ are Murray-von Neumann equivalent in $M_N(A^\dagger)$ for some $N \in \mathbb N$. By Lemma \ref{20-11-22a} we may assume that $r \in P_\infty(\mathcal M_\tau^\dagger)$. It follows then from Lemma \ref{19-11-22c} that 
$$
\tau^\dagger \otimes \Tr_N(e\oplus f' \oplus r) = \tau^\dagger \otimes \Tr_N(e'\oplus f \oplus r)
$$ 
and hence by linearity that
$$
\tau^\dagger\otimes \Tr_n(e) -\tau^\dagger\otimes \Tr_n(f) = \tau^\dagger\otimes \Tr_n(e') -\tau^\dagger\otimes \Tr_n(f').
$$

\qed

\bigskip

By definition $K_0(A)$ is a subgroup of $K_0(A^\dagger)$; the kernel of map $K_0(A^\dagger) \to K_0(\mathbb C)$. The map $\tau^\dagger_*$ of Theorem \ref{19-11-22d} gives therefore rise, by restriction, to a homomorphism 
$$
\tau_* : K_0(A) \to \mathbb R
$$
such that
$$
\tau_*([e]-[f]) = \tau^\dagger \otimes \Tr_n(e) - \tau^\dagger \otimes \Tr_n(f)
$$
when $e,f \in M_n(\mathcal M_\tau^\dagger)$ and $[e]-[f] \in K_0(A)$.

Recall that $K_0(A)$ comes equipped with an order given by the semi-group
$$
K_0(A)_+ := \left\{[e] : \ e \in P_\infty(A) \right\} ,
$$
cf. III 6 in \cite{B}.

\begin{lemma}\label{20-11-22hx} Let $e$ be a projection in $M_k(A)$. Then $e \in M_k(\mathcal M_\tau)$.
\end{lemma}
\begin{proof} Since $\tau_k$ is densely defined by Lemma \ref{29-09-22x} there is an $x \in \mathcal M_{\tau_k}^+$ such that $\left\|e -x \right\| < \frac{1}{2}$. Then
$$
\frac{1}{2} e \leq e\left(1 + (x-e)\right)e = exe \in \mathcal M_{\tau_k} .
$$
Hence $\tau_k(e) \leq 2 \tau_k(exe) < \infty$, showing that $e \in \mathcal M_{\tau_k}$. This completes the proof because $\mathcal M_{\tau_k} = M_k(\mathcal M_\tau)$ by (f) of Lemma \ref{26-09-22x}.

\end{proof}

Since $\tau^\dagger\otimes \Tr_n = \tau_n$ on $M_n(\mathcal M_\tau)$ it follows from Lemma \ref{20-11-22hx} and Theorem \ref{19-11-22d} that
\begin{equation}\label{20-11-22e}
\tau_* ([e]-[f]) = \tau_n(e)  - \tau_n(f)
\end{equation}
when $e$ and $f$ are projections in $M_n(A)$. In particular, it follows that $\tau_*$ is a positive homomorphism in the sense that
$$
\tau_*(K_0(A)_+) \subseteq [0,\infty) .
$$


\subsection{On lower semi-continuity of traces}

A part of the next proposition is known and follows from Lemma 3.1 in \cite{ERS} and Remark 2.27 (iv) in \cite{BK}.

\begin{prop}\label{02-01-23} Let $\tau$ a densely defined trace on the $C^*$-algebra $A$. There is a densely defined lower semi-continuous trace $\tilde{\tau}$ on $A$ such that 
\begin{enumerate}
\item[(a)] $\tilde{\tau}(a) \leq \tau(a)$ for all $a \in A^+$, and
\item[(b)] if $\phi$ is a lower semi-continuous weight on $A$ such that $\phi(a) \leq \tau(a)$ for all $a \in A^+$, then $\phi(a) \leq \tilde{\tau}(a)$ for all $a \in A^+$.
\end{enumerate}
The trace $\tilde{\tau}$ has the property that $\tau_* = {\tilde{\tau}}_*$ on $K_0(A)$.
\end{prop}
\begin{proof}  Assume first that $A$ is separable. Let $\{d_n\}$ be a sequence in $A$ with the properties specified in Lemma \ref{19-11-22fx}. Then $d_n^2 = d_{n+1}d_n^2d_{n+1} \leq d_{n+1}^2$ and hence
$$
\tau(d_na^*ad_n) = \tau(ad_n^2a^*) \leq \tau(ad_{n+1}^2a^*) = \tau(d_{n+1}a^*ad_{n+1})
$$
for all $n$ and all $a \in A$. We can therefore define $\tilde{\tau} : A^+ \to [0,\infty]$ such that
$$
\tilde{\tau}(a) = \lim_{n \to \infty} \tau(d_nad_n) = \sup_n \tau(d_nad_n) .
$$
Note that $\tilde{\tau}$ is a weight.
It follows from Lemma \ref{18-11-22gx} that $\tau(d_n \ \cdot \ d_n)$ is (the restriction to $A^+$ of) a bounded positive linear functional on $A$ and hence also that $\tilde{\tau}$ is lower semi-continuous. This is used in the following calculation. Let $a \in A$. Then
\begin{align*}
&\tilde{\tau}(aa^*) = \lim_{n \to \infty}\tilde{\tau}(ad_n^2a^*) = \lim_{n \to \infty} \lim_{k \to \infty} \tau(d_kad_n^2a^*d_k) \\
& =  \lim_{n \to \infty} \lim_{k \to \infty} \tau(d_na^*d_k^2ad_n)  =  \lim_{n \to \infty}  \tau(d_na^*ad_n)  = \tilde{\tau}(a^*a) ;
\end{align*}
i.e. $\tilde{\tau}$ is a trace. Since 
$$
\tilde{\tau}(a^*a) = \lim_{n \to \infty} \tau(d_na^*ad_n) = \lim_{n \to \infty} \tau(ad_n^2a^* ) \leq \tau(aa^*) = {\tau}(a^*a),
$$ 
we conclude that $\tilde{\tau}(b) \leq \tau(b)$ for all $b \in A^+$, i.e. (a) holds. To see that so does (b), let $\phi$ be a lower semi-continuous weight on $A$ such that $\phi(a) \leq \tau(a)$ for all $a \in A^+$. Let $b \in A$. Then
\begin{align*}
& \phi(b^*b) = \lim_{n \to \infty} \phi(b^*d_n^2b) \leq \lim_{n \to \infty} \tau(b^*d_n^2b) \\
&= \lim_{n \to \infty} \tau(d_n bb^*d_n) = \tilde{\tau}(bb^*) = \tilde{\tau}(b^*b) .
\end{align*}

 To prove that $\tau_* = {\tilde{\tau}}_*$ on $K_0(A)$ we prove that $\tau^\dagger_* = {\tilde{\tau}}^\dagger_*$ on $K_0(A^\dagger)$, which will suffice. To this end note that $\mathcal M_{\tau}^\dagger \subseteq \mathcal M_{\tilde{\tau}}^\dagger$ since $\tilde{\tau} \leq \tau$. By Theorem \ref{19-11-22d} it suffices therefore to show that $\tau^\dagger \otimes \Tr_n(e) = \tilde{\tau}^\dagger \otimes \Tr_n(e)$ when $e$ is a projection in $M_n(\mathcal M_{\tau}^\dagger)$. For this note that it follows from the first part of the proof of Lemma \ref{19-11-22c} that there is a $k \in \mathbb N$ and a projection $e' \in M_n(D_k^\dagger)$, where $D_k := d_kAd_k$, such that $\tau^\dagger \otimes \Tr_n(e) =\tau^\dagger \otimes \Tr_n(e')$ and $\tilde{\tau}^\dagger \otimes \Tr_n(e) =\tilde{\tau}^\dagger \otimes \Tr_n(e')$. This completes the proof because $\tau^\dagger \otimes \Tr_n(e') = \tilde{\tau}^\dagger \otimes \Tr_n(e')$ since $\tau$ and $\tilde{\tau}$ clearly agree on $D_k$.
 
 Consider then the general case. It follows from Lemma \ref{18-11-22gx} that $\tau$ restricts to a densely defined trace on any separable $C^*$-subalgebra $D$ of $A$. It follows therefore from the preceding that there is a lower semi-continuous trace $\tilde{\tau}_D$ on $D$ with the desired properties relative to $\tau|_D$. It will be enough to show that
 $$
 {\tilde{\tau}_{D_2}}|_{D_1} = \tilde{\tau}_{D_1}
 $$
 when $D_1$ and $D_2$ are separable $C^*$-subalgebras of $A$ such that $D_1 \subseteq D_2$. To this end observe that there are approximate units $\{d_n\}$ for $D_1$ and $\{d'_n\}$ for $D_2$ with the properties specified in Lemma \ref{19-11-22fx} such that 
 $$
 \lim_{n \to \infty} \tau(d_nad_n) = \tilde{\tau}_{D_1}(a) \ \ \forall a \in D_1^+
 $$
 and
 $$
  \lim_{n \to \infty} \tau(d'_nad'_n) = \tilde{\tau}_{D_2}(a) \ \ \forall a \in D_2^+ .
 $$
Let $a \in D_1^+$. Since $\tau(d_n \ \cdot \ d_n)$ is bounded, 
\begin{align*}
& \tau(d_nad_n) = \lim_{k \to \infty} \tau(d_n a^{\frac{1}{2}}{d'_k}^2 a^{\frac{1}{2}}d_n)  = \lim_{k \to \infty} \tau(d'_ka^{\frac{1}{2}}d_n^2a^{\frac{1}{2}}d'_k) \\
& = \tilde{\tau}_{D_2}(a^{\frac{1}{2}}d_n^2a^{\frac{1}{2}}) .
\end{align*} 
Since $\tilde{\tau}_{D_2}$ is lower semi-continuous this gives
$$
\tilde{\tau}_{D_1}(a) = \lim_{n \to \infty} \tau(d_nad_n) = \lim_{n \to \infty}  \tilde{\tau}_{D_2}(a^{\frac{1}{2}}d_n^2a^{\frac{1}{2}}) = \tilde{\tau}_{D_2}(a) .
$$
We can now define $\tilde{\tau} : A^+ \to [0,\infty]$ by
$$
\tilde{\tau}(a) := \tilde{\tau}_D(a),
$$
where $D \subseteq A$ is any separable $C^*$-subalgebra of $A$ containing $a$. It is straightforward to check that $\tilde{\tau}$ has the stated properties.
\end{proof}

The most famous traces that are not lower semi-continuous are the traces on the $C^*$-algebra $\mathbb K$ of compact operators on an infinite dimensional separable Hilbert space constructed by J. Dixmier in \cite{Di}. Dixmiers traces are used in non-commutative geometry as described in Chapter IV of \cite{Co2} and Dixmiers construction has been investigated and extended by many mathematicians. An almost defining feature of these traces is that they invariably vanish on the operators of finite rank, and hence in particular on all projections in $\mathbb K$. It follows from this that if we apply Proposition \ref{02-01-23} with a Dixmier trace in the role of $\tau$, the resulting lower semi-continuous trace $\tilde{\tau}$ is the zero trace. In contrast, if we apply Proposition \ref{02-01-23} to the trace $\tau : \mathbb K^+ \to [0,\infty]$ defined by
$$
\tau(a) = \begin{cases} \Tr(a) &  \text{when $ a$ has finite rank}\\ \infty  & \text{when $a$ does not have finite rank,}\end{cases}
$$
the result is the usual trace $\Tr$ on $\mathbb K$.

\begin{example}\label{27-10-23}
Let $X$ be a locally compact Hausdorff space, $A$ a $C^*$-algebra and $\tau$ a non-zero densely defined trace on $A$. Set 
$$
B:= C_0(X) \otimes A = C_0(X,A);
$$ 
the $C^*$-algebra of $A$-valued continuous functions on $X$ vanishing in norm at infinity. When we assume that $X$ is second countable and not compact, we can choose a sequence $\{x_i\}_{i=1}^\infty$ of points in $X$ such that $i \neq j \Rightarrow x_i \neq x_j$ and $\lim_{i \to \infty} x_i = \infty$ in the sense that for every compact subset $K \subseteq X$ there is a $N_K \in \mathbb N$ such that $x_i \notin K$ when $i \geq N_K$. Let $\{t_i\}_{i=1}^\infty$ be a sequence of positive real numbers such that $\lim_{i \to \infty} t_i = \infty$, and let finally $\xi \in \beta \mathbb N \backslash \mathbb N$ be a free ultrafilter in $\mathbb N$ which we think of as an element of the Stone-\v{C}ech remainder of $\mathbb N$. Define $\mu : B^+ \to [0,\infty]$ such that
$$
\mu(b) = \lim_{N \to \xi} \frac{1}{N} \sum_{n=1}^N t_n \tau(b(x_n))
$$
when 
$\left\{\frac{1}{N} \sum_{n=1}^N t_n \tau(b(x_n))\right\}_{N \in \mathbb N}$
 is bounded and 
 $$\mu(b) = \infty
 $$ when $\left\{\frac{1}{N} \sum_{n=1}^N t_n \tau(b(x_n))\right\}_{N \in \mathbb N}$ is not bounded. It is clear that $\mu$ is a densely defined trace on $B$; it is zero on elements of $C_0(X,A)$ that take values in $\mathcal M_\tau$ and have compact support. To see that $\mu$ is not zero and not lower semi-continuous, choose an element $a \in A^+$ such that $0 < \tau(a) < \infty$. Choose continuous functions $f_n \in C_0(X)$, $n \in \mathbb N \backslash \{0\}$, with mutually disjoint compact supports such that $0 \leq f_n \leq 1$, $f_n(x_n) = 1$ and $f_n(x_i) = 0, \ i \neq n$. Then
 $$
 g := \sum_{n=1}^\infty t_n^{-1}f_n \otimes a  \in C_0(X,A)^+,
 $$
and $\mu(g) = \tau(a)$. Hence $\mu$ is not zero. Since $g = \lim_{L \to \infty} \sum_{n=1}^L t_n^{-1} f_n \otimes a$ and
$$
\mu\left(  \sum_{n=1}^L t_n^{-1} f_n \otimes a\right) = 0
$$
for all $L$, we see that $\mu$ is not lower semi-continuous.
\end{example}

The algebra $A$ in Example \ref{27-10-23} is not simple and all the examples of traces that are not lower semi-continuous exhibited in Example \ref{27-10-23} share the property with the Dixmier traces that they vanish on projections. The next example does not have any of these two deficiencies, although the underlying idea is the same.

\begin{example}\label{23-10-23a}

 Let $B(\mathbb N,\mathbb Q)$ denote the set of bounded sequences of rational numbers.
Let $G$ be the additive subgroup of $B(\mathbb N,\mathbb Q)$ consisting of the sequences $\{a_n\}_{n=1}^\infty$ in $\mathbb Q$ with the property that there is an $N \in \mathbb N$ such that 
$$
a_i =  \frac{q}{i^2}
$$ 
for some $q \in \mathbb Q$ and all $i \geq N$. Set
$$
G^+ = \left\{ \{a_n\}_{n=1}^\infty \in G : \ a_n > 0 \ \forall n \in \mathbb N \right\} \cup \{0\} .
$$
Then $(G,G^+)$ is a countable simple dimension group. Set 
$$
\Sigma := \left\{  \{a_n\}_{n=1}^\infty \in G^+ : \ a_n < 1 \ \ \forall n \in \mathbb N \right\} .
$$
Then $\Sigma$ is a scale in $(G,G^+)$ and there is therefore a simple AF algebra $A$ such that
$$
(K_0(A),K_0(A)^+,\Sigma(A)) = (G,G^+,\Sigma) .
$$
This follows from \cite{E1} and \cite{EHS}. If necessary, see Proposition 1.4.5 in \cite{R}.

For each $k \in \mathbb N$ there is a bounded trace $\tau_k$ on $A$ such that 
$$
{\tau_k}_*\left( \{a_n\}_{n=1}^\infty \right) = a_k. 
$$
Note that $\left\|\tau_k\right\| \leq 1$.
Let $\omega$ be a free-ultrafilter in $\mathbb N$, which we consider as an element of the Stone-\^Cech remainder of $\mathbb N$. Define
$$
\tau : A^+ \to [0,\infty]
$$
such that
$$
\tau(a) =  \lim_{N \to \omega} \frac{1}{N} \sum_{k=1}^N \tau_k(a)k^2
$$ 
when $\left\{\frac{1}{N} \sum_{k=1}^N \tau_k(a)k^2\right\}_{N=1}^\infty$ is bounded and
$$
\tau(a) = \infty
$$
when $\left\{\frac{1}{N} \sum_{k=1}^N \tau_k(a)k^2\right\}_{N=1}^\infty$ is unbounded. Then $\tau$ is clearly a trace. To show that $\tau$ is non-zero fix $j \in \mathbb N$ and let $p_j \in A$ be a projection such that
$$
\tau_k(p_j) = \begin{cases} \frac{1}{k^2}, & \ k > j, \\ \frac{1}{2}, & \ k = j, \\ 2^{-j}, & \ k < j . \end{cases}
$$ 
Then
$$
\tau_k\left( \sum_{j=1}^\infty \frac{p_j}{j^2} \right)k^2 = \frac{1}{2}+  \sum_{j <k} \frac{1}{j^2} + \sum_{j > k} \frac{2^{-j} k^2}{j^2}  ,
$$ 
and hence
$$
\tau\left(\sum_{j=1}^\infty \frac{p_j}{j^2}\right) = \frac{1}{2} + \sum_{j=1}^\infty \frac{1}{j^2} =\frac{1}{2} + \frac{\pi^2}{6} .
$$
We note that for $N > M$,
\begin{align*}
& \sum_{k=1}^N \tau_k\left( \sum_{j=1}^M \frac{p_j}{j^2}\right)k^2 \\
& = \sum_{k=M+1}^N \sum_{j=1}^M \frac{1}{j^2} + \sum_{k=1}^M \tau_k\left( \sum_{j=1}^M \frac{p_j}{j^2}\right)k^2 \\
& =\sum_{k=M+1}^N \sum_{j=1}^M \frac{1}{j^2} + \sum_{k=1}^M \left( (\sum_{j=1}^{k-1} \frac{1}{j^2}) + \frac{1}{2} + ( \sum_{j=k+1}^M \frac{2^{-j}}{j^2} k^2) \right) .
\end{align*}
Since 
$$
(\sum_{j=1}^{k-1} \frac{1}{j^2}) + \frac{1}{2} + ( \sum_{j=k+1}^M \frac{2^{-j}}{j^2} k^2)
$$
is bounded by a constant $C$ so that 
$$
\sum_{k=1}^M \left( (\sum_{j=1}^{k-1} \frac{1}{j^2}) + \frac{1}{2} + ( \sum_{j=k+1}^M \frac{2^{-j}}{j^2} k^2) \right)
$$
is bounded by $MC$, and since
$$
\sum_{k=M+1}^N \sum_{j=1}^M \frac{1}{j^2} = (N-M) \sum_{j=1}^M \frac{1}{j^2}
$$
we find that 
$$
\tau\left(  \sum_{j=1}^M \frac{p_j}{j^2}\right) = \sum_{j=1}^M \frac{1}{j^2} .
$$
Thus 
$$
\tau\left(\sum_{j=1}^M \frac{p_j}{j^2}\right) \leq  \tau\left(\sum_{j=1}^\infty \frac{p_j}{j^2}\right)-\frac{1}{2} 
$$
for all $M$. Since $\lim_{M \to \infty}\sum_{j=1}^M \frac{p_j}{j^2} = \sum_{j=1}^\infty \frac{p_j}{j^2}$ this shows that $\tau$ is not lower semi-continuous. To see that $\tau$ is densely defined note that for any projection $q \in A$ there is an $N \in \mathbb N$ such that 
$$
\tau_k(q) \leq \tau_k(N p_1) \leq N\tau_k\left(\sum_{j=1}^\infty \frac{p_j}{j^2}\right)
$$ for all $k \in \mathbb N$, implying that
$$
\tau(q) \leq N \tau\left( \sum_{j=1}^\infty \frac{p_j}{j^2}\right) = N(\frac{1}{2} + \frac{\pi^2}{6}) < \infty .
$$ 
Since every element of $A^+$ can be approximated in norm by a linear combination of projections with non-negative coefficients, it follows that $\tau$ is densely defined. Note that $\tau$ does not vanish on any non-zero projection. Indeed, when $e$ is a non-zero projection in $A$ there is a positive rational number $q$ such that $\tau_k(e) = q \frac{1}{k^2}$ for all sufficient large $k$, leading to the conclusion that $\tau(e) =q > 0$. 

Like the constructions in Example \ref{27-10-23} the preceding can also be varied. To see one possible variation, let $0 < \epsilon < 1$ and define 
$$
\tau_\epsilon : A^+ \to [0,\infty]
$$
such that
$$
\tau_\epsilon(a) =  \lim_{N \to \omega} \frac{1}{N} \sum_{k=1}^N \tau_k(a)k^{1+\epsilon}
$$ 
when $\left\{\frac{1}{N} \sum_{k=1}^N \tau_k(a) k^{1+\epsilon} \right\}_{N=1}^\infty$ is bounded and
$$
\tau_\epsilon(a) = \infty
$$
when $\left\{\frac{1}{N} \sum_{k=1}^N \tau_k(a)k^{1+\epsilon}\right\}_{N=1}^\infty$ is unbounded. Then
\begin{align*}
&\tau_k\left(\sum_{j=1}^\infty j^{-(1+\epsilon)} p_j\right)k^{1+\epsilon} 
 = \sum_{j=1}^\infty \left(\frac{k}{j}\right)^{1+\epsilon} \tau_k(p_j) \\
 & = \frac{1}{2} + \sum_{j < k} \frac{1}{k^2}  \left(\frac{k}{j}\right)^{1+\epsilon} + \sum_{j > k } 2^{-j}  \left(\frac{k}{j}\right)^{1+\epsilon} \\
 & =\frac{1}{2} + \frac{k^{1+\epsilon}}{k^2}\sum_{j < k} \frac{1}{ j^{1+\epsilon}} + \sum_{j > k } 2^{-j}  \left(\frac{k}{j}\right)^{1+\epsilon} .
\end{align*}
Since $\lim_{k \to \infty} \sum_{j > k } 2^{-j}  = \lim_{k \to \infty}  \frac{k^{1+\epsilon}}{k^2}\sum_{j < k} \frac{1}{ j^{1+\epsilon}} =0$ we have that
$$
\lim_{k \to \infty} \tau_k\left(\sum_{j=1}^\infty j^{-(1+\epsilon)} p_j\right)k^{1+\epsilon}  = \frac{1}{2} ,
$$
and hence that 
$$\tau_\epsilon\left(\sum_{j=1}^\infty j^{-(1+\epsilon)} p_j\right) = \frac{1}{2} .
$$
This shows that $\tau_\epsilon$ is non-zero. To see that $\tau_\epsilon$ is not lower semi-continuous, densely defined and not proportional to $\tau$ note that $\tau_\epsilon(e) = 0$ for all projections $e \in A$.

\end{example}

Of course, the sum of the usual trace and a Dixmier trace on $\mathbb K$ gives also examples of densely defined traces on a simple $C^*$-algebra which are not lower semi-continuous and do not vanish on projections.

\end{document}